\documentclass[conference]{IEEEtran}
\IEEEoverridecommandlockouts

\usepackage{cite}
\usepackage{amsmath,amssymb,amsfonts}
\usepackage{algorithmic}
\usepackage{graphicx}
\usepackage{multirow}
\usepackage{booktabs}
\usepackage{textcomp}
\usepackage{xcolor}
\def\BibTeX{{\rm B\kern-.05em{\sc i\kern-.025em b}\kern-.08em
    T\kern-.1667em\lower.7ex\hbox{E}\kern-.125emX}}

\usepackage{geometry}

\newgeometry{
  top=19.1mm,
  bottom=19.1mm,
  left=19.1mm,
  right=19.1mm
}

\usepackage[hidelinks]{hyperref}

\graphicspath{{figures/}}

\usepackage[ruled,linesnumbered]{algorithm2e}
\usepackage{algorithmic}

\usepackage{amsmath}
\usepackage{cases}
\usepackage{amsthm}

\newtheorem{theorem}{Theorem}

\newtheorem{remark}{Remark}

\usepackage{mathtools}
\usepackage{mathrsfs}
\usepackage{newtxmath}
\renewcommand\mathcal[1]{\text{\usefont{OMS}{cmsy}{m}{n}#1}}

\newcommand\difffrac[3][1]{
	\ifnum #1=1
		\frac{\mathrm{d} #2}{\mathrm{d} #3}
	\else
	\frac{{\mathrm{d}}^{#1} #2}{\mathrm{d} #3^{#1}}
	\fi
}

\newcommand\R{{\mathbb{R}}}

\newcommand\st{{\mathrm{s.t.}}}
\newcommand\tf{{t_\text{f}}}
\newcommand\f{{\text{f}}}

\renewcommand\vector[1]{\boldsymbol{#1}}

\newcommand\vt{{\vector{t}}}

\newcommand\vx{{\vector{x}}}

\newcommand\vlambda{{\vector{\lambda}}}
\newcommand\veta{{\vector{\eta}}}

\newcommand\vphi{{\vector{\phi}}}
\newcommand\vzeta{{\vector{\zeta}}}

\newcommand\vzero{{\vector{0}}}

\usepackage{accents}
\renewcommand\underbar[1]{{\underaccent{\bar}{#1}}}

\usepackage{xcolor}

\begin{document}

\title{\vspace{6.3mm}Time-Optimal Switching Surfaces for Triple Integrator Under Full Box Constraints
\thanks{This paper has been accepted by American Control Conference, 2026.}
\thanks{The authors would like to thank Tianxiang Lu for his valuable suggestions on Gr\"obner basis and findings on the discontinuity of the proper position in our previous work \cite{wang2025time}. This work was supported in part by the National Natural Science Foundation of China under Grant 624B2077, and in part by the National Key Research and Development Program of China under Grant 2023YFB4302003.}
\thanks{All authors are with the State Key Laboratory of Tribology in Advanced Equipment, Department of Mechanical Engineering, Tsinghua University, Beijing, 100084, China, and also with the Beijing Key Laboratory of Transformative High-end Manufacturing Equipment and Technology, Department of Mechanical Engineering, Tsinghua University, Beijing 100084, China.}
\thanks{\textsuperscript{*}Corresponding author: Chuxiong Hu (email: cxhu@tsinghua.edu.cn).}
}

\author{
    Yunan Wang, Chuxiong Hu\textsuperscript{*}, and Zhao Jin
}

\maketitle

\begin{abstract}
    Time-optimal control for triple integrator under full box constraints is a fundamental problem in the field of optimal control, which has been widely applied in the industry. However, scenarios involving asymmetric constraints, non-stationary boundary conditions, and active position constraints pose significant challenges. This paper provides a complete characterization of time-optimal switching surfaces for the problem, leading to novel insights into the geometric structure of the optimal control. The active condition of position constraints is derived, which is absent from the literature. An efficient algorithm is proposed, capable of planning time-optimal trajectories under asymmetric full constraints and arbitrary boundary states, with a 100\% success rate. Computational time for each trajectory is within approximately 10\,$\mu\text{s}$, achieving a 5-order-of-magnitude reduction compared to optimization-based baselines.
\end{abstract}


\section{Introduction}\label{sec:introduction}

Time-optimal control for triple integrator has been widely applied in robotic motion \cite{berscheid2021jerk,kong2025efficient}, numerical control machining \cite{wang2026online,wang2025consistency}, and so on. As a fundamental problem in the field of optimal control, time-optimal control under full box constraints, i.e., jerk-limited, has been widely investigated. However, a complete and efficient solution remains absent, especially under asymmetric constraints, non-stationary boundary conditions, and active position constraints. Specifically, the analytical characterization of switching surfaces can significantly advance theoretical understanding and improve computational efficiency of the solution \cite{he2020time}.

Based on Pontryagin's maximum principle (PMP) \cite{pontryagin1987mathematical}, time-optimal control for triple integrator under full box constraints can be characterized by the bang-bang and singular (BBS) control laws, giving rise to the well-known S-curve \cite{du2015complete}. The problem without state constraints was solved analytically \cite{bartolini2002time} and could be represented by Gr\"obner basis \cite{patil2015computation}. Kr{\"o}ger \cite{kroger2011opening} solved the problem without specifying terminal acceleration. He et al. \cite{he2020time} derived complete switching surfaces for the problem with stationary terminal states, where position constraints are inactive. Berscheid and Kr{\"o}ger \cite{berscheid2021jerk} developed Ruckig, an efficient algorithm for the problem with box constraints and arbitrary boundary states. However, the jerk constraint is required to be symmetric, and position constraints are not supported in the community version. Our previous work \cite{wang2025time} addressed the problem based on augmented switching law (ASL). However, the geometric structure of optimal solutions, i.e., switching surfaces, remains uncharacterized, and the active condition of the position constraints is absent. Furthermore, the algorithm in \cite{wang2025time} is not complete, i.e., it may fail for certain boundary states.

This paper aims to completely characterize the switching surfaces and establish an efficient algorithm for the time-optimal control problem under full box constraints. For the first time, time-optimal switching surfaces for triple integrator with active position constraints are provided, revealing novel geometric insights into the structure of optimal control. Leveraging the derived switching surfaces, an efficient algorithm is proposed, capable of planning time-optimal trajectories under asymmetric full constraints and arbitrarily given boundary states within approximately 10\,$\mu\text{s}$ of computational time. Furthermore, the established algorithm significantly outperforms optimization-based baselines regarding computational efficiency, numerical stability, and time-optimality. The implementation of the proposed algorithm is available at \cite{MIMGithub2025}.

\section{Preliminaries}\label{sec:preliminaries}

\subsection{Hamiltonian Analysis}

Consider the following optimal control problem of order $n$, where $\vx=(x_k(t))_{k=1}^n\in\R^n$ is the state vector, $u=u(t)\in\R$ is the control input, and $\tf>0$ is the free terminal time. The system dynamics \eqref{eq:problem_dynamics} takes the chain-of-integrator form. In box constraints \eqref{eq:problem_constraints}, a stationary state and control are feasible, i.e., $\underbar{\vx}=(\underbar{x}_k)_{k=1}^{n}<\vzero<\bar{\vx}=(\bar{x}_k)_{k=1}^n$ and $\underbar{u}<0<\bar{u}$ hold. Specifically, if $n=2,\,3$, then the system \eqref{eq:problem_dynamics} is called the double and triple integrator, respectively.
\begin{subequations}\label{eq:problem}
    \begin{align}
        \min_{(u,\tf)}\quad&J=\tf=\int_{0}^{\tf}\mathrm{d}t,\\
        \st\quad&\dot{x}_1=u,\,\,\dot{x}_k=x_{k-1},\,\,k=2,\dots,n,\label{eq:problem_dynamics}\\
        &\underbar{u}\leq u \leq \bar{u},\,\,\underbar{x}_k\leq x_k \leq \bar{x}_k,\,\,k=1,\dots,n,\label{eq:problem_constraints}\\
        &\vx(0)=\vx_0,\,\,\vx(\tf)=\vx_f.
    \end{align}
\end{subequations}

Denote the costate vector and multiplier vector by $\vlambda=(\lambda_k(t))_{k=1}^n\in\R^n$ and $\veta=(\eta_k^\pm(t))_{k=1}^n\in\R^{2n}_{+}$, respectively. The Hamiltonian $\mathcal{H}$ and Lagrangian $\mathcal{L}$ are defined as
\begin{subequations}
    \begin{align}
        &\mathcal{H}(\vx,\vlambda,u,t)=\lambda_0+\lambda_1u+\sum_{k=1}^{n-1}\lambda_{k+1}x_k\equiv0,\\
        &\mathcal{L}=\mathcal{H}+\sum_{k=1}^n\eta_k^+(x_k-\bar{x}_k)+\sum_{k=1}^n\eta_k^-(\underbar{x}_k-x_k),
    \end{align}
\end{subequations}
where $\eta_k^+(x_k-\bar{x}_k)=0$ and $\eta_k^-(x_k-\underbar{x}_k)=0$ hold almost everywhere (a.e.) in $\left[0,\tf\right]$. Furthermore, $\lambda_0\geq0$ is constant. It holds a.e. that $\dot{\vlambda}=-\frac{\partial\mathcal{L}}{\partial\vx}$, i.e.,
\begin{equation}
    \dot{\lambda}_n=0,\,\dot{\lambda}_k=-\left(\lambda_{k+1}+\eta_k^+-\eta_k^-\right),\,k=1,\dots,n-1.
\end{equation}
The Lagrangian form of problem \eqref{eq:problem} implies that $\mathcal{H}\equiv0$ holds along the optimal trajectory. The junction condition \cite{maurer1977optimal} implies that if a state constraint $h(\vx)\leq0$ switches between active and inactive at time $t_i$,
\begin{equation}\label{eq:junction_condition}
    \exists\zeta(t_i)\leq0,\,\st\,\vlambda(t_i^+)-\vlambda(t_i^-)=\zeta(t_i)\frac{\partial h(\vx)}{\partial\vx}.
\end{equation}
In other words, if $x_k\leq\bar{x}_k$ or $x_k\geq\underbar{x}_k$ switch between active and inactive at $t_i$, then $\lambda_k(t_i^+)\leq\lambda_k(t_i^-)$ or $\lambda_k(t_i^+)\geq\lambda_k(t_i^-)$ hold, respectively. Denote $\vzeta=(\zeta(t_i))$, where each $t_i$ is a junction time. The nontriviality condition is given that
\begin{equation}\label{eq:nontriviality_condition}
    (\lambda_0,\vlambda(t),\veta(t),\vzeta)\not=\vzero,\,\forall t\in\left[0,\tf\right].
\end{equation}

In this paper, the above analysis can be applied as necessary conditions of optimality for problem \eqref{eq:problem}.

\subsection{Theoretical Results in \cite{wang2025time}}\label{sec:theoretical_results_in_tac}
This section summarizes the theoretical results presented in our previous work \cite{wang2025time}. While \cite{wang2025time} assumes symmetric constraints for simplicity, i.e., $\bar{u}=-\underbar{u}$ and $\bar{x}_k=-\underbar{x}_k$, the theoretical results can be directly extended to asymmetric cases in this paper. According to PMP \cite{pontryagin1987mathematical}, the optimal control takes the form of a BBS control law, i.e.,
\begin{equation}\label{eq:PMP}
    u(t)=\begin{dcases}
        \bar{u},&\text{if } \lambda_1(t)<0,\\
        0,&\text{if } \lambda_1(t)=0,\\
        \underbar{u},&\text{if } \lambda_1(t)>0.
    \end{dcases}
\end{equation}
All arcs are classified into the following $2n$ types:
\begin{itemize}
    \item $\bar{0}$: $u\equiv\bar{u}$.
    \item $\underbar{0}$: $u\equiv\underbar{u}$.
    \item $\bar{k}$ ($k<n$): $x_k\equiv\bar{x}_k$; hence, $u\equiv0$, and $\forall j<k$, $x_j\equiv0$.
    \item $\underbar{k}$ ($k<n$): $x_k\equiv\underbar{x}_k$; hence, $u\equiv0$, and $\forall j<k$, $x_j\equiv0$.
\end{itemize}
Specifically, $\bar{0}$- and $\underbar{0}$-arcs are called unconstrained arcs, while others are called constrained arcs.

The tangent marker, which was rarely investigated in problem \eqref{eq:problem} prior to \cite{wang2025time}, is a common behavior in 3rd-order problems with position constraints. For problem \eqref{eq:problem} of order $n\leq3$, tangent markers only exist when $n=3$, and fall into the following two types:
\begin{itemize}
    \item $(\bar{3},2)$: $x_3$ is tangent to $\bar{x}_3$ at $t_i$, i.e., $\dot{x}_3(t_i)=\bar{x}_3$, $x_2(t_i)=0$, and $x_1(t_i)\leq0$.
    \item $(\underbar{3},2)$: $x_3$ is tangent to $\underbar{x}_3$ at $t_i$, i.e., $\dot{x}_3(t_i)=\underbar{x}_3$, $x_2(t_i)=0$, and $x_1(t_i)\geq0$.
\end{itemize}
In \eqref{eq:junction_condition}, the junction time can be either the connection of two adjacent arcs or a tangent marker.

Furthermore, if problem \eqref{eq:problem} is feasible, then the optimal solution is unique in an a.e. sense.

\section{Time-Optimal Switching Surfaces}

This section investigates problem \eqref{eq:problem} of order $n=3$. Section \ref{sec:double_integrator} first considers the double integrator. Then, Sections \ref{subsec:switching_surface_3order_inf_position} and \ref{subsec:full_box_constraints} construct the switching surfaces of problem \eqref{eq:problem} of order $n=3$ without and with position constraints, respectively. Based on the results of switching surfaces, an efficient algorithm is proposed in Section \ref{subsec:algorithm}.

\subsection{Degenerate Case: Reduction to Double Integrator}\label{sec:double_integrator}
Following the approach in \cite{wang2025time}, the first step is to construct the 2nd-order optimal-trajectory manifold $\mathcal{F}_2$. Specifically, this corresponds to problem \eqref{eq:problem} of order $n = 2$.

\begin{figure}[!t]
    \centering
    \includegraphics[width=\linewidth]{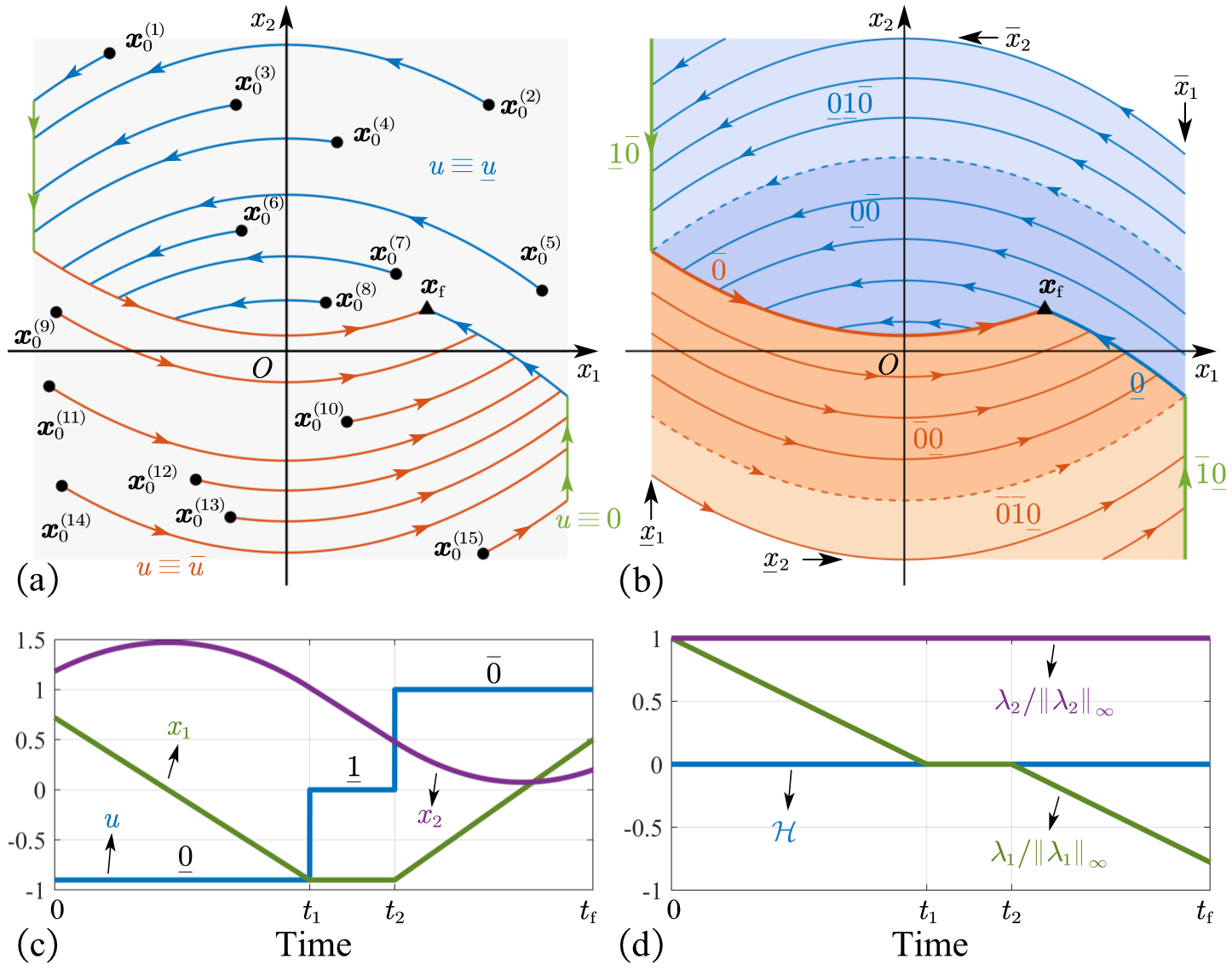}
    \caption{Illustration of the double integrator. (a) Time-optimal trajectories from various initial states $\vx_0$ to a fixed terminal state $\vx_\f$. (b) 2nd-order optimal-trajectory manifolds. (c-d) The state and costate profiles of the time-optimal trajectory between $\vx_0^{(2)}=(0.72,1.182)$ and $\vx_\f$. In this example, we assign $\vx_\f=(0.5,0.2)$, $-0.9\leq u\leq1$, and $(-0.9,-1)\leq\vx\leq(1,1.5)$.}
    \label{fig:2order_manifold}
\end{figure}

As shown in Fig. \ref{fig:2order_manifold}(a), we choose an arbitrary initial state $\vx_0$ and fix $\vx_\f\in\R^2$. If the problem is feasible, then the time-optimal trajectory from $\vx_0$ to $\vx_\f$ can be represented as the augmented switching law (ASL), i.e., the sequence of arcs and tangent markers. The set of initial states $\vx_0$ that share the same ASL forms a smooth manifold, as shown in Fig. \ref{fig:2order_manifold}(b).

Consider a numerical example in Fig. \ref{fig:2order_manifold}(c-d). The ASL is $\underbar{0}\underbar{1}\bar{0}$, i.e., $\vx_0$ moves along a $\underbar{0}$-, $\underbar{1}$-, and $\bar{0}$-arcs sequentially and finally reaches $\vx_\f$. Conditions of optimality in Section \ref{sec:preliminaries}, like PMP and $\mathcal{H}\equiv0$, are satisfied.

Specifically, the $\bar{0}$-manifold, from which the time-optimal trajectory to $\vx_\f$ consists a single $\bar{0}$-arc, is
\begin{equation}
    x_1=x_{\f1}-\bar{u}t,\,x_2=x_{\f2}-x_{\f1}t+\frac{1}{2}\bar{u}t^2,\,t\in\left(0,\frac{x_{\f1}-\underbar{x}_1}{\bar{u}}\right).
\end{equation}
The $\underbar{1}\bar{0}$-manifold, from which the time-optimal trajectory to $\vx_\f$ consists a $\underbar{1}$-arc and a $\bar{0}$-arc sequentially, is
\begin{equation}
    x_1=\underbar{x}_1,\,x_2\in\left(x_{\f2}+\frac{\underbar{x}_1^2-x_{\f1}^2}{2\bar{u}},\bar{x}_2\right).
\end{equation}
The $\underbar{0}$- and $\bar{1}\underbar{0}$-manifolds can be derived similarly.  Without loss of generality, this section considers the case where $(0,\bar{x}_2)\in\underbar{0}\underbar{1}\bar{0}$ and $(0,\underbar{x}_2)\in\bar{0}\bar{1}\underbar{0}$.

Based on the above switching curves, the optimal strategy can be provided in a closed-loop form. If the current state $\vx(t)$ lies on $\bar{0}$, $\underbar{0}$, or $\underbar{1}\bar{0}\cup\bar{1}\underbar{0}$, then let the current control be $u(t)=\bar{u}$, $\underbar{u}$, or $0$, respectively; otherwise, if $\vx(t)$ is above or below $\bar{0}\cup\underbar{0}$, then let $u(t)=\underbar{u}$ or $\bar{u}$, respectively. In other words, the closed-loop control strategy is given by
\begin{equation}\label{eq:closed_loop_control_2order}
    u\left(\vx\right)=\begin{dcases}
        \bar{u},&\text{if } \vx\in\bar{0}\cup\bar{0}\underbar{0}\cup\bar{0}\bar{1}\underbar{0},\\
        \underbar{u},&\text{if } \vx\in\underbar{0}\cup\underbar{0}\bar{0}\cup\underbar{0}\underbar{1}\bar{0},\\
        0,&\text{if } \vx\in\underbar{1}\bar{0}\cup\bar{1}\underbar{0}.
    \end{dcases}
\end{equation}
In \eqref{eq:closed_loop_control_2order}, each ASL refer to the relative interiors of the corresponding manifold.

\begin{remark}
    In the double integrator, the controls in $\underbar{0}\bar{0}$ and $\underbar{0}\underbar{1}\bar{0}$ are both $\underbar{u}$. However, in the triple integrator, these two sets correspond to distinct analytical expressions as switching surfaces. Furthermore, the $\underbar{1}\bar{0}$ and $\bar{1}\underbar{0}$ manifolds are allowed to be empty, where the optimal control derived in this section remains valid.
\end{remark}

\subsection{Triple Integrator without Position Constraints}\label{subsec:switching_surface_3order_inf_position}

This section considers problem \eqref{eq:problem} of order $n=3$ without position constraints, i.e., $\underbar{x}_3=-\infty$ and $\bar{x}_3=\infty$, the switching surfaces of which have not yet to be completely characterized in the literature. Specifically, Patil et al. \cite{patil2015computation} did not consider state constraints. Yury \cite{yury2016quasi} could not achieve time-optimality. He et al. \cite{he2020time} only considered the case where $\vx_\f=\vzero$.

Similar to Section \ref{sec:double_integrator}, $\vx_\f$ and constraints are fixed. As shown in Fig. \ref{fig:3order_manifold_inf_position}(a), time-optimal trajectories from various initial states $\vx_0$ to $\vx_\f$ can be summarized as a switching surface-based control strategy. As shown in Fig. \ref{fig:3order_manifold_inf_position}(b), time-optimal switching surfaces for triple integrator without position constraints are constructed based on the results in second integrator as follows.

Define $\vphi:\R^3\times\R\times\R\to\R^3,\left(\vx,u,t\right)\mapsto(x_1+ut,x_2+x_1t+\frac12ut^2,x_3+x_2t+\frac12x_1t^2+\frac16ut^3)$. Starting from $\vx_\f$, the $\bar{0}$- and $\underbar{1}\bar{0}$-curves are constructed through backward integration:
\begin{align}
    &\bar{0}=\left\{\vphi\left(\vx_\f,\bar{u},-t\right):\,t\in\left[0,\frac{x_{\f1}-\underbar{x}_1}{\bar{u}}\right]\right\},\\
    &\underbar{1}\bar{0}=\left\{\vphi\left(\vx_{\bar{0}},0,-t\right):\,t\in\left[0,\frac{x_{\bar{0},2}-\bar{x}_2}{\underbar{x}_1}\right]\right\}\subset\left\{x_1=\underbar{x}_1\right\},
\end{align}
where $\vx_{\bar{0}}=(x_{\bar{0},k})_{k=1}^3\triangleq\vphi\left(\vx_\f,\bar{u},\frac{\underbar{x}_1-x_{\f1}}{\bar{u}}\right)$ denotes an end of the $\bar{0}$-curve, as highlighted in Fig. \ref{fig:3order_manifold_inf_position}(b).

Without loss of generality, we still assume that $(0,\bar{x}_2)$ and $(0,\underbar{x}_2)$ lie in the projection of $\underbar{0}\underbar{1}\bar{0}$ and $\bar{0}\bar{1}\underbar{0}$ onto the $x_1x_2$-plane, respectively. In other words, it holds that
\begin{equation}\label{eq:010_plus_condition}
    x_{\underbar{1}\bar{0},2}\triangleq\bar{x}_2+\frac{\underbar{x}_1^2}{2\underbar{u}}\geq x_{\bar{0},2}.
\end{equation}
$\underbar{0}\bar{0}$- and $\underbar{0}\underbar{1}\bar{0}$-surfaces can be constructed through backward integration from points in $\bar{0}$- and $\underbar{1}\bar{0}$-curves, respectively:
\begin{align}
    &\underbar{0}\bar{0}=\left\{\vphi\left(\hat{\vx},\underbar{u},-t\right):\,\hat{\vx}\in\bar{0},\,t\in\left[0,\frac{\hat{x}_{1}-\bar{x}_1}{\underbar{u}}\right]\right\},\label{eq:01_plus}\\
    &\underbar{0}\underbar{1}\bar{0}=\left\{\vphi\left(\hat{\vx},0,-t\right):\,\hat{\vx}\in\underbar{1}\bar{0},\,t\in\left[0,t_{\underbar{0}\underbar{1}\bar{0}}(\hat{\vx})\right]\right\}.\label{eq:010_plus}
\end{align}
In \eqref{eq:010_plus}, $t_{\underbar{0}\underbar{1}\bar{0}}(\vx)=\frac{x_{1}-\bar{x}_1}{\underbar{u}}$ if $x_{2}\geq x_{\underbar{1}\bar{0},2}$; otherwise, $t_{\underbar{0}\underbar{1}\bar{0}}(\vx)=\frac{x_1+\sqrt{x_1^2-2\underbar{u}(x_2-\bar{x}_2)}}{\underbar{u}}$.

\begin{figure}[!t]
    \centering
    \includegraphics[width=\linewidth]{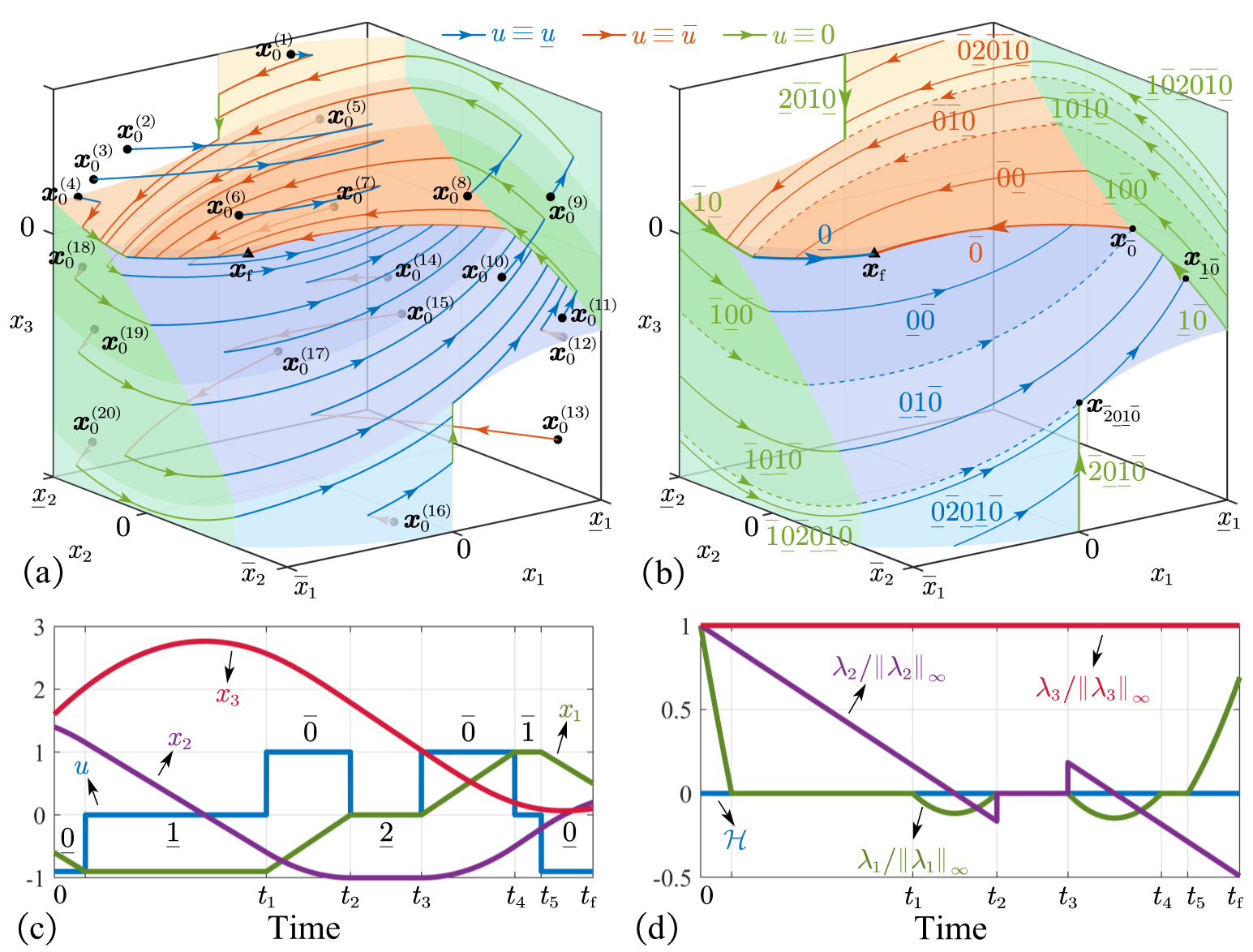}
    \caption{An example of the triple integrator without position constraints. (a) Time-optimal trajectories from different initial states $\vx_0$ to a fixed terminal state $\vx_\f$. (b) Switching manifolds. (c-d) The state and costate profiles of the time-optimal trajectory between $\vx_0^{(9)}=(-0.6,1.4,1.6)$ and $\vx_\f$, where the ASL is $\underline{0}\underline{1}\bar{0}\underline{2}\bar{0}\bar{1}\underline{0}$. In this example, we assign $\vx_\f=(0.5,0.2,0.1)$, $-0.9\leq u\leq1$, and $(-0.9,-1,-\infty)\leq\vx\leq(1,1.5,\infty)$.}
    \label{fig:3order_manifold_inf_position}
\end{figure}

Consider the intersection of the $\underbar{0}\bar{0}$-surface and the plane $\left\{x_1=\bar{x}_1\right\}$. The $\bar{1}\underbar{0}\bar{0}$-surface is constructed through backward integration from the above intersection curve:
\begin{equation}
    \bar{1}\underbar{0}\bar{0}=\left\{\vphi\left(\hat{\vx},0,-t\right):\,\hat{\vx}\in\underbar{0}\bar{0},\,\hat{x}_1=\bar{x}_1,\,t\in\left[0,\frac{\hat{x}_{2}-\underbar{x}_2}{\bar{x}_1}\right]\right\}.
\end{equation}
Similarly, the $\bar{1}\underbar{0}\underbar{1}\bar{0}$-surface is
\begin{equation}
    \bar{1}\underbar{0}\underbar{1}\bar{0}=\left\{\vphi\left(\hat{\vx},0,-t\right):\,\hat{\vx}\in\underbar{0}\underbar{1}\bar{0},\,\hat{x}_1=\bar{x}_1,\,t\in\left[0,\frac{\hat{x}_{2}-\underbar{x}_2}{\bar{x}_1}\right]\right\}.
\end{equation}

For a constrained arc where $x_2\equiv\bar{x}_2$, it holds that $x_1\equiv0$ and $u\equiv0$. In other words, the phase $x_2\equiv\bar{x}_2$ only exists in the line $\bar{2}\triangleq\left\{x_2=\bar{x}_2,\,x_1=0\right\}$. Assumption \eqref{eq:010_plus_condition} implies that the $\bar{2}$-line intersects the $\underbar{0}\underbar{1}\bar{0}$-surface at a point $\vx_{\bar{2}\underbar{0}\underbar{1}\bar{0}}\triangleq\vphi\left(\vx_{\underbar{1}\bar{0}},\underbar{u},-\frac{x_{\underbar{1}\bar{0}}}{\underbar{u}}\right)$ where $\vx_{\underbar{1}\bar{0}}\triangleq\vphi\left(\vx_{\bar{0}},0,-\frac{x_{\bar{0},2}-x_{\underbar{1}\bar{0},2}}{\underbar{x}_1}\right)$. Starting from $\vx_{\bar{2}\underbar{0}\underbar{1}\bar{0}}$, the $\bar{2}\underbar{0}\underbar{1}\bar{0}$-curve is
\begin{equation}
    \bar{2}\underbar{0}\underbar{1}\bar{0}=\left\{\vx:\,x_1=x_2=0,\,x_3\leq x_{ \bar{2}\underbar{0}\underbar{1}\bar{0},3}\right\}.
\end{equation}

The $\underbar{0}\bar{2}\underbar{0}\underbar{1}\bar{0}$- and $\bar{1}\underbar{0}\bar{2}\underbar{0}\underbar{1}\bar{0}$-surfaces are constructed by performing backward integration from the $\bar{2}\underbar{0}\underbar{1}\bar{0}$-curve as:
\begin{align}
    &\underbar{0}\bar{2}\underbar{0}\underbar{1}\bar{0}=\left\{\vphi\left(\hat\vx,\underbar{u},-t\right):\,\hat\vx\in\bar{2}\underbar{0}\underbar{1}\bar{0},\,t\in\left[0,-\frac{\bar{x}_1}{\underbar{u}}\right]\right\},\\
    &\bar{1}\underbar{0}\bar{2}\underbar{0}\underbar{1}\bar{0}=\left\{\vphi\left(\hat\vx,0,-t\right):\,\hat{\vx}\in\underbar{0}\bar{2}\underbar{0}\underbar{1}\bar{0},\,\hat{x}_1=\bar{x}_1,\,t\in\left[0,\frac{\hat{x}_2-\underbar{x}_2}{\bar{x}_1}\right]\right\}.
\end{align}

Based on the constructed switching surfaces, the state space $\R^3$ can be partitioned into regions, each corresponding to a unique ASL. The volumetric regions where $u\equiv\bar{u}$ are shown in Fig. \ref{fig:3order_volume_inf_position}. For example, $\bar{0}\underbar{0}\bar{0}$ is the volumetric region generated by extruding the surface $\underbar{0}\bar{0}$ along the generatrix $u\equiv\bar{0}$ backwardly, i.e.,
\begin{equation}
    \bar{0}\underbar{0}\bar{0}=\left\{\vphi\left(\hat{\vx},\bar{u},-t\right):\,\hat{\vx}\in\underbar{0}\bar{0},\,t\in\left[0,\frac{\hat{x}_{1}-\underbar{x}_1}{\bar{u}}\right]\right\}.
\end{equation}
In other words, $\forall \vx_0\in\bar{0}\underbar{0}\bar{0}$, the optimal trajectory from $\vx_0$ to $\vx_\f$ is represented by the ASL $\bar{0}\underbar{0}\bar{0}$, i.e., a sequence of a $\bar{0}$-arc, a $\underbar{0}$-arc, and a $\bar{0}$-arc. The region $\bar{0}\underbar{0}\underbar{1}\bar{0}$ is defined as
\begin{equation}
    \bar{0}\underbar{0}\underbar{1}\bar{0}=\left\{\vphi\left(\hat{\vx},\bar{u},-t\right):\,\hat{\vx}\in\underbar{0}\underbar{1}\bar{0},\,t\in\left[0,t_{\bar{0}\bar{1}\underbar{0}}(\hat{\vx})\right]\right\},
\end{equation}
where $t_{\bar{0}\bar{1}\underbar{0}}(\vx)=\frac{x_{1}-\underbar{x}_1}{\bar{u}}$ if $x_{2}\geq \underbar{x}_2+\frac{\bar{x}_1^2}{2\bar{u}}$; otherwise, $t_{\bar{0}\bar{1}\underbar{0}}(\vx)=\frac{x_1-\sqrt{x_1^2-2\bar{u}(x_2-\underbar{x}_2)}}{\bar{u}}$. Similar analysis can be applied to other 3D regions whose ASLs are ended with $\underbar{0}$. 

\begin{remark}
    If \eqref{eq:010_plus_condition} does not hold, then the $\bar{2}\underbar{0}\underbar{1}\bar{0}$-curve should be substituted by the $\bar{2}\underbar{0}\bar{0}$-curve, so as $\underbar{0}\bar{2}\underbar{0}\underbar{1}\bar{0}$- and $\bar{1}\underbar{0}\bar{2}\underbar{0}\underbar{1}\bar{0}$-surfaces. Similar analysis can be applied to the $\underbar{2}\bar{0}\bar{1}\underbar{0}$-curve. Furthermore, if some constraints are not assigned, then the corresponding manifolds should be omitted. For example, if $\bar{x}_1=\infty$, then those manifolds with ASLs containing $\bar{1}$ should be omitted.
\end{remark}

\begin{figure}[!t]
    \centering
    \includegraphics[width=\linewidth]{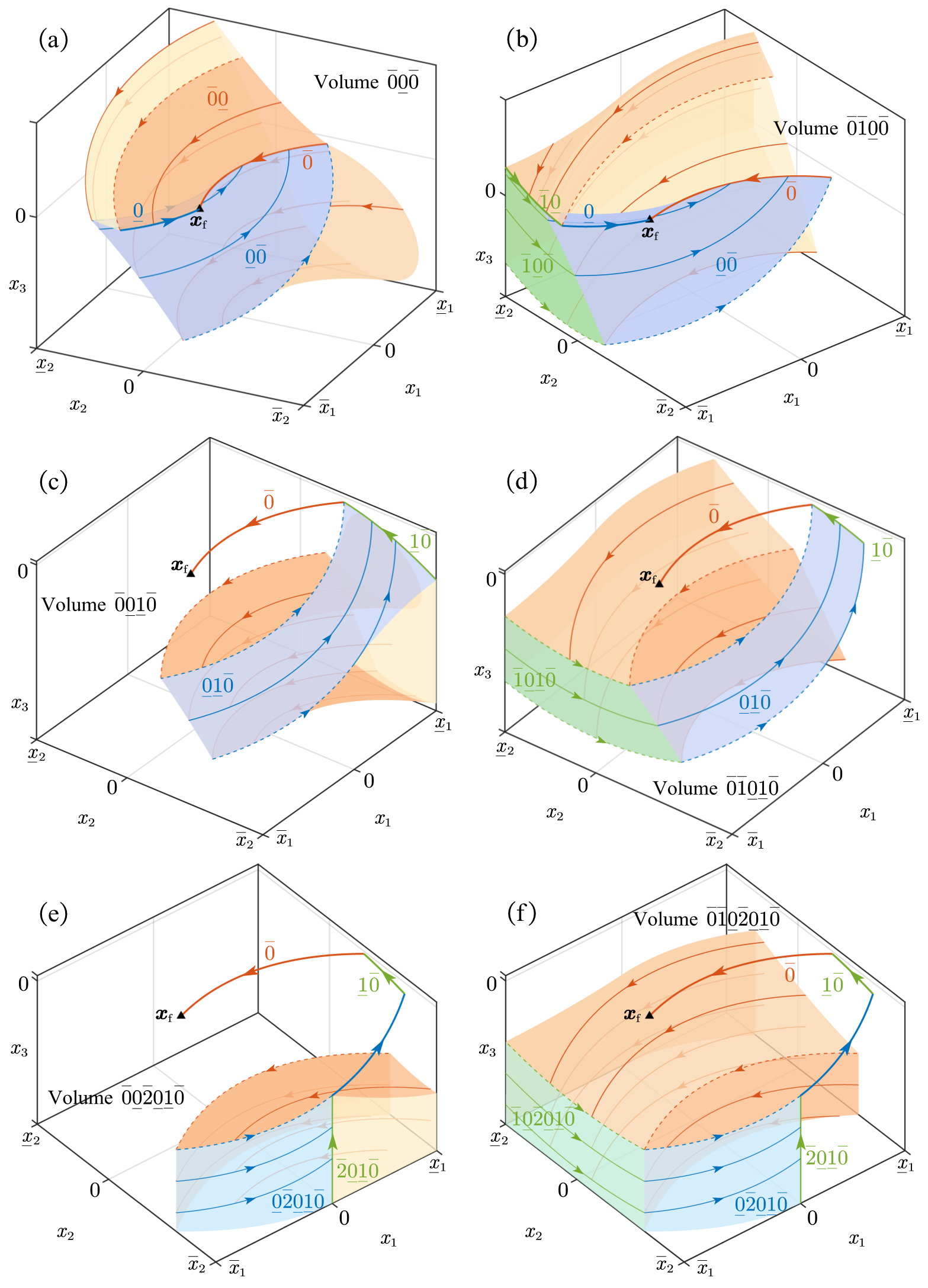}
    \caption{Partition of the state space $\R^3$ by ASLs. The terminal state and constraints are the same as those in Fig. \ref{fig:3order_manifold_inf_position}.}
    \label{fig:3order_volume_inf_position}
\end{figure}

Similar to \eqref{eq:closed_loop_control_2order}, the closed-loop control strategy for the triple integrator without position constraints can be derived as follows:
\begin{equation}\label{eq:control_law_3order_inf_position}
    u(\vx)=\begin{dcases}
        \bar{u},&\text{if } \vx\in\text{relint}(\mathfrak{L}),\,\mathfrak{L}\text{ begins with }\bar{0},\\
        \underbar{u},&\text{if } \vx\in\text{relint}(\mathfrak{L}),\,\mathfrak{L}\text{ begins with }\underbar{0},\\
        0,&\text{otherwise},
    \end{dcases}
\end{equation}
where $\text{relint}(\mathfrak{L})$ denotes the relative interior of the manifold whose ASL is $\mathfrak{L}$. It should be noted that the switching surfaces $\bar{0}\underbar{0}\cup\bar{0}\bar{1}\underbar{0}\cup\underbar{0}\bar{0}\cup\underbar{0}\underbar{1}\bar{0}$ is \textit{NOT} single-valued with respective to (w.r.t.) $\left(x_1,x_2\right)$, and therefore cannot be expressed in the explicit form $x_3=x_3(x_1,x_2)$ at extremal positions. 

\begin{theorem}\label{thm:3order_inf_position}
    The control law \eqref{eq:control_law_3order_inf_position} is complete for the triple integrator without position constraints. In other words, $\forall\,\vx_0,\vx_\f\in\R^3$, if the optimal control problem from $\vx_0$ to $\vx_\f$ is feasible, then the trajectory given by \eqref{eq:control_law_3order_inf_position} satisfies PMP.
\end{theorem}

\begin{proof}
    According to \cite{berscheid2021jerk}, the time-optimal trajectory should be of at most 7 phases. Denote $\mathcal{A}_3^+=$\{$\bar{0}\underbar{0}\bar{0}$, $\bar{0}\bar{1}\underbar{0}\bar{0}$, $\bar{0}\underbar{0}\underbar{1}\bar{0}$, $\bar{0}\bar{1}\underbar{0}\underbar{1}\bar{0}$, $\bar{0}\underbar{0}\bar{2}\underbar{0}\underbar{1}\bar{0}$, $\bar{0}\bar{1}\underbar{0}\bar{2}\underbar{0}\underbar{1}\bar{0}$\} and $\mathcal{A}_3^-=$\{$\underbar{0}\bar{0}\underbar{0}$, $\underbar{0}\underbar{1}\bar{0}\underbar{0}$, $\underbar{0}\bar{0}\bar{1}\underbar{0}$, $\underbar{0}\underbar{1}\bar{0}\bar{1}\underbar{0}$, $\underbar{0}\bar{0}\underbar{2}\bar{0}\bar{1}\underbar{0}$, $\underbar{0}\underbar{1}\bar{0}\underbar{2}\bar{0}\bar{1}\underbar{0}$\}. Then, the ASL can be $\mathfrak{L}\in\mathcal{A}_3^+\cup\mathcal{A}_3^-$.

    Arbitrarily assign the initial state $\vx_0\in\R^3$, where the optimal control problem is feasible with an ASL $\mathfrak{L}\in\mathcal{A}$. According to the backward integration process, switching surfaces in this section can drive the state $\vx$ from $\vx_0$ to $\vx_\f$ with an ASL $\mathfrak{L}'\in\mathcal{A}$. Specifically, $\mathfrak{L}'=\mathfrak{L}$ holds unless the trajectory induced by $\mathfrak{L}$ is intercepted by other switching surfaces induced by $\mathfrak{L}'$.

    Denote the solution induced by $\mathfrak{L}'$ as $\left(\vx(t),u(t),t_\f\right)$. Next, it is proved that $\left(\vx(t),u(t),t_\f\right)$ satisfies PMP. Without loss of generality, we consider the case where $\mathfrak{L}'=\underbar{0}\underbar{1}\bar{0}\underbar{2}\bar{0}\bar{1}\underbar{0}$. Denote the terminal time of the 6 arcs by $\left\{t_i\right\}_{i=1}^7$ which increases monotonically. Let $\lambda_0=1$ and $\lambda_3(t)\equiv-\frac{1}{\underbar{x}_2}$. Then, for $t\in(t_3,t_4)$, $\mathcal{H}\equiv0$ holds since $x_2\equiv\underbar{x}_2$, $x_1\equiv0$, and $u\equiv0$. Let
    \begin{equation}
        -\underbar{x}_2\lambda_1(t)=\begin{dcases}
            \left(t-t_1\right)\left(t+t_1-t_2-t_3\right),&t\in\left(0,t_1\right),\\
            \left(t-t_2\right)\left(t-t_3\right),&t\in\left(t_2,t_3\right),\\
            \left(t-t_4\right)\left(t-t_5\right),&t\in\left(t_4,t_5\right),\\
            \left(t-t_6\right)\left(t+t_6-t_5-t_4\right),&t\in\left(t_6,t_\f\right),\\
            0,&\text{otherwise}.\\
        \end{dcases}
    \end{equation}
    \begin{equation}
        \underbar{x}_2\lambda_2(t)=\begin{dcases}
            t-(t_2+t_3)/2,&t\in\left(0,t_3\right),\\
            0,&t\in\left(t_3,t_4\right),\\
            t-(t_4+t_5)/2,&t\in\left(t_4,t_\f\right),
        \end{dcases}
    \end{equation}

    It can be verified that PMP \eqref{eq:PMP} holds. Specifically, $\lambda_1$ is continuous, while $\lambda_2$ jumps increasingly at junction times $t_3$ and $t_4$. Therefore, the junction condition \eqref{eq:junction_condition} is also satisfied. Note that $\lambda_0\not=0$; hence, the nontriviality condition \eqref{eq:nontriviality_condition} holds.

    For other $\mathfrak{L}'\in\mathcal{A}$, a similar construction of the costate $\vlambda(t)$ can be applied, where PMP and other conditions of optimality are satisfied.
\end{proof}

\begin{remark}
    In the proof of Theorem \ref{thm:3order_inf_position}, we can ``design'' the costate according to the ASL $\mathfrak{L}'\in\mathcal{A}$ and the switching time of each arc. In this way, switching surfaces constructed through backward integration induce a complete control strategy satisfying PMP. Once $\vx$ enters a low-order switching manifold, not only does the control $u$ switch, but also the state $\vx$ moves along the switching manifold. 

    However, the above analysis cannot be applied to problem \eqref{eq:problem} of order $n\geq4$ due to the inability of constrained arcs to directly connect unconstrained arcs \cite{wang2025chattering}. The state $\vx$ would not move along but cross some switching manifolds when $\vx$ enters them. In higher-order cases, chattering phenomena may occur where the control switches between $\bar{u}$ and $\underbar{u}$ infinitely over a finite duration.
\end{remark}

A numerical example is provided in Fig. \ref{fig:3order_manifold_inf_position}(c-d), where the ASL is $\underbar{0}\underbar{1}\bar{0}\underbar{2}\bar{0}\bar{1}\underbar{0}$.

\subsection{Triple Integrator with Full Box Constraints}\label{subsec:full_box_constraints}

This section considers problem \eqref{eq:problem} of order $n=3$ with full box constraints, i.e., $-\infty<\underbar{x}_3<\bar{x}_3<\infty$. To the best of the authors' knowledge, the switching surfaces of triple integrator with position constraints have never been characterized in the literature. Before our previous work \cite{wang2025time}, only Ruckig Pro, a closed-source commercial software \cite{Ruckig2025Online}, claimed the ability to deal with position constraints in triple integrator. \cite{wang2025time} pointed out that position constraints in triple integrator induce \textit{tangent markers} and solved the problem in a decoupled way. However, the switching surfaces were not characterized.

\begin{figure}[!t]
    \centering
    \includegraphics[width=\linewidth]{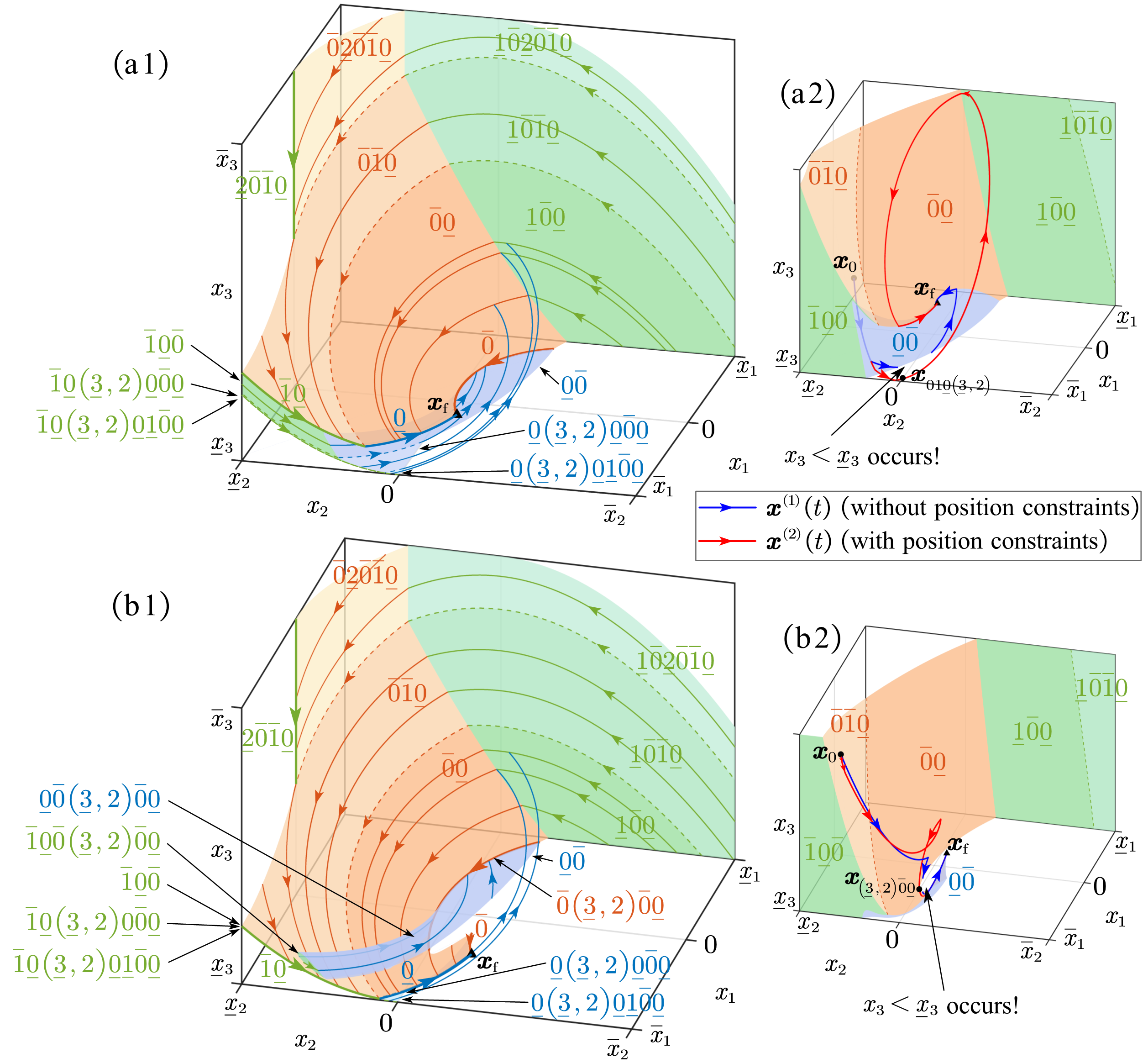}
    \caption{Switching manifolds of the triple integrator with full constraints, where $(-0.9,-1,0)\leq\vx\leq(1,1.5,2.4)$ and $-0.9\leq u\leq1$. (a) $\vx_\f=(0.5,0.2,0.2)$. (b) $\vx_\f=(0.5,0.3,0.075)$. In (a1) and (b1), switching manifolds consider the position constraints. In (a2) and (b2), switching manifolds do not consider position the position constraints. For each case, two trajectories, i.e., $\vx^{(1)}(t)$ and $\vx^{(2)}(t)$, are planned under the same boundary conditions but without and with considering position constraints, respectively. The initial states in (a2) and (b2) are $\vx_0=(0.5,-0.6292,0.2636)$ and $(0.98,-0.5944,0.2722)$, respectively.}
    \label{fig:3order_manifold_constrained_position}
\end{figure}

Consider the problems illustrated in Fig. \ref{fig:3order_manifold_constrained_position}, where position constraints $0\leq x_3\leq 2.4$ are introduced. As shown in Fig. \ref{fig:3order_manifold_constrained_position}(a2) and (b2), initial states $\vx_0$ lie ``above'' the switching surfaces in Section \ref{subsec:switching_surface_3order_inf_position} without considering position constraints; hence, the control should begin with $u=\bar{u}$. In this way, the planned trajectories $\vx^{(1)}(t)$ are infeasible since they cross the constraint boundary $\left\{x_3=\underbar{x}_3\right\}$. 

To minimize the degradation of time-optimality caused by position constraints, the control should begin with $u=\bar{u}$ until the profile planned by the strategy in Section \ref{subsec:switching_surface_3order_inf_position} is ``exactly'' feasible. In other words, the resulting profile should be tangent to the constraint boundary $\left\{x_3=\underbar{x}_3\right\}$, satisfying $x_3=\underbar{x}_3$, $x_2=0$, and $x_1\geq0$, which is denoted by the tangent marker $(\underline{3},2)$ in Section \ref{sec:theoretical_results_in_tac}. In Fig. \ref{fig:3order_manifold_constrained_position}(a2), $\vx_0$ first moves to $\vx_{\bar{0}\bar{1}\underbar{0}(\underbar{3},2)}\in\left\{x_3=\underbar{x}_3,\,x_2=0,\,x_1\geq0\right\}$ through $\bar{0}$-, $\bar{1}$-, and $\underbar{0}$-arcs. Then, the state moves from $\vx_{\bar{0}\bar{1}\underbar{0}(\underbar{3},2)}$ to $\vx_\f$ based on the control strategy in Section \ref{subsec:switching_surface_3order_inf_position}. In this way, the planned trajectory $\vx^{(2)}(t)$ is feasible and satisfies PMP, as shown in Fig. \ref{fig:3order_trajectory_constrained_position}(a). Similarly, in Fig. \ref{fig:3order_manifold_constrained_position}(b2), $\vx_\f$ moves backward to $\vx_{(\underbar{3},2)\bar{0}\underbar{0}}\in\left\{x_3=\underbar{x}_3,\,x_2=0,\,x_1\geq0\right\}$ through $\underbar{0}$- and $\bar{0}$-arcs. Next, the state moves from $\vx_0$ to $\vx_{(\underbar{3},2)\bar{0}\underbar{0}}$. The resulting trajectory $\vx^{(2)}(t)$ is shown in Fig. \ref{fig:3order_trajectory_constrained_position}(b).

Switching surfaces for triple integrator with full box constraints are shown in Fig. \ref{fig:3order_manifold_constrained_position}(a1) and (b1). The behavior of switching surfaces is much more intricate than that without position constraints. Similar to Section \ref{subsec:switching_surface_3order_inf_position}, the switching surfaces are constructed through backward integration from the terminal state $\vx_\f$, where the range of integration time is determined by all constraints. By comparing Fig. \ref{fig:3order_manifold_constrained_position}(a1) and Fig. \ref{fig:3order_manifold_inf_position}, the upper boundary of surface $\underbar{1}\bar{0}\underbar{2}\bar{0}\bar{1}\underbar{0}$-surface is redefined as a composite of a linear segment and a quadratic curve due to the position constraint $x_3\leq\bar{x}_3$. The $\underbar{0}\bar{0}$- and $\bar{1}\underbar{0}\bar{0}$-surfaces are also redefined as follows. $\underbar{0}(\underbar{3},2)\underbar{0}\bar{0}\underbar{0}$-, $\underbar{0}(\underbar{3},2)\underbar{0}\underbar{1}\bar{0}\underbar{0}$-, $\bar{1}\underbar{0}(\underbar{3},2)\underbar{0}\bar{0}\underbar{0}\text{-,}$ and $\bar{1}\underbar{0}(\underbar{3},2)\underbar{0}\underbar{1}\bar{0}\underbar{0}$-surfaces are introduced. For example, $\bar{1}\underbar{0}(\underbar{3},2)\underbar{0}\underbar{1}\bar{0}\underbar{0}$-surface is defined as \{$\vx_0$: $\exists\vt\in\R^6$, $t_i>0$, s.t. $\vx_i=\vphi(\vx_{i-1},u_i,t_i)$, $\vx_6=\vx_\f$, $\vector{u}=(0,\underbar{u},\underbar{u},0,\bar{u},\underbar{u})$, $x_{0,1}=\bar{x}_1$, $x_{2,1}>0$, $x_{2,2}=0$, $x_{2,3}=\underbar{x}_3$, $x_{3,1}=\underbar{x}_1$, and the induced trajectory is feasible\}. In other words, once the state enters the $\bar{1}\underbar{0}(\underbar{3},2)\underbar{0}\underbar{1}\bar{0}\underbar{0}$-surface, the state is tangent to \{$x_3=\underbar{x}_3$, $x_2=0$, $x_1\geq0$\} through $\bar{1}$- and $\underbar{0}$-arcs, and then the state reaches $\vx_\f$ through $\underbar{0}$-, $\underbar{1}$-, $\bar{0}$-, and $\underbar{0}$-arcs. In Fig. \ref{fig:3order_manifold_constrained_position}(b1), initial states in some regions should move to a fixed point $\vx_{(\underline{3},2)\bar{0}\underbar{0}}$ first, and then move to $\vx_\f$ through a fixed trajectory. These optimal trajectories induce switching manifolds $\bar{0}(\underline{3},2)\bar{0}\underbar{0}$, $\underbar{0}\bar{0}(\underline{3},2)\bar{0}\underbar{0}$, and $\bar{1}\underbar{0}\bar{0}(\underline{3},2)\bar{0}\underbar{0}$.

Denote $\mathcal{A}_2^+=$\{$\bar{0}\underbar{0}$, $\bar{0}\bar{1}\underbar{0}$\} and $\mathcal{A}_2^-=$\{$\underbar{0}\bar{0}$, $\underbar{0}\underbar{1}\bar{0}$\}. The following theorem summarizes allowed ASLs with position constraints.

\begin{figure}[!t]
    \centering
    \includegraphics[width=\linewidth]{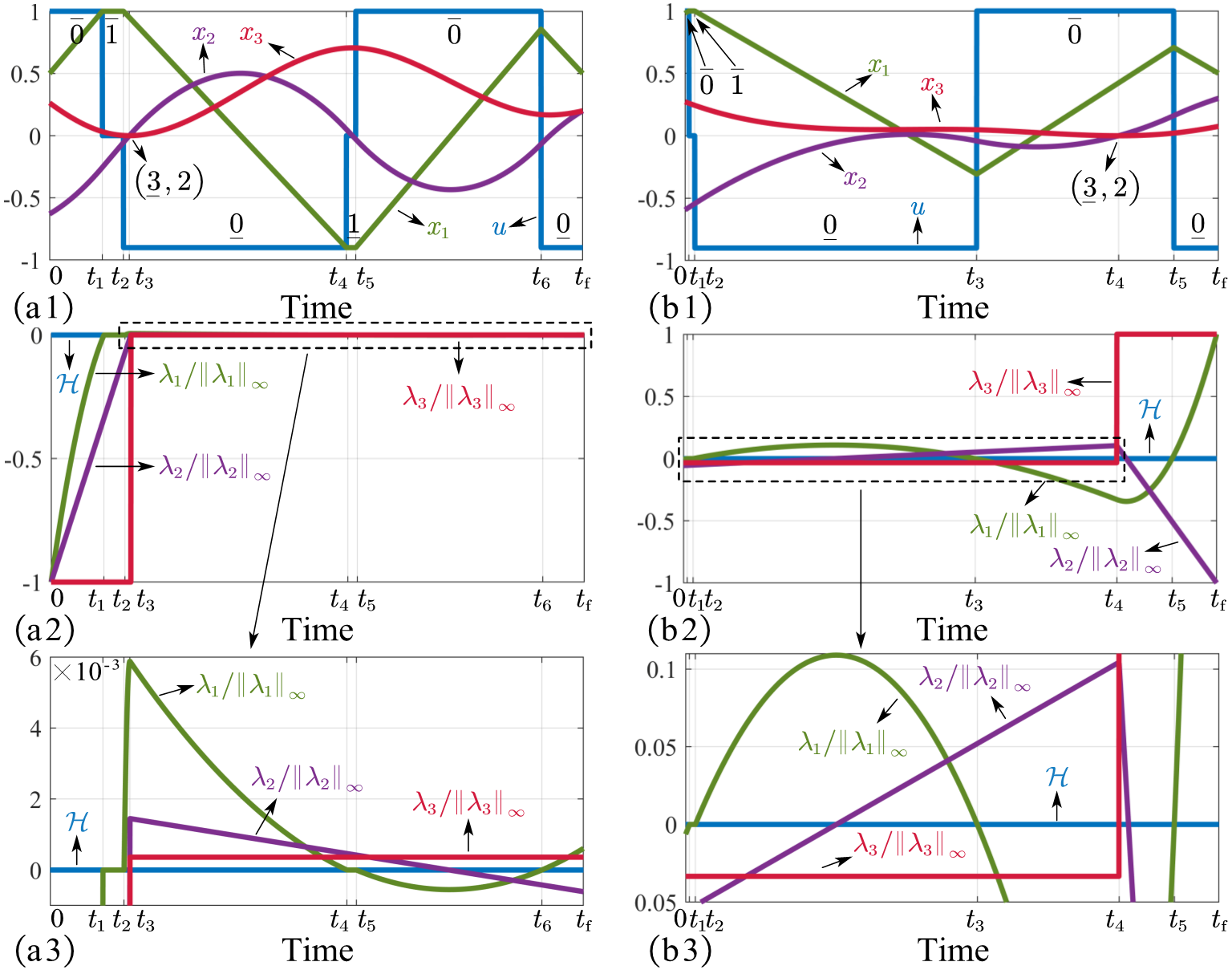}
    \caption{State and costate profiles of time-optimal trajectories with the tangent marker $(\underline{3},2)$. In this figure, (a) and (b) correspond to $\vx^{(2)}(t)$ Fig. \ref{fig:3order_manifold_constrained_position}(a) and (b) with ASLs $\bar{0}\bar{1}\underline{0}(\underline{3},2)\underline{0}\underline{1}\bar{0}\underline{0}$ and $\bar{0}\bar{1}\underline{0}\bar{0}(\underline{3},2)\bar{0}\underline{0}$, respectively.}
    \label{fig:3order_trajectory_constrained_position}
\end{figure}

\begin{theorem}\label{thm:3order_full_box_constraints}
    Time-optimal trajectories for triple integrator with full box constraints can be represented by ASLs of the following forms:
    \begin{subequations}\label{eq:ASL_3_2}
        \begin{align}
            &\hspace{-2mm}\dots(\bar{3},2)\beta_2^+(\underbar{3},2)\beta_1^-(\bar{3},2)\alpha^+(\underbar{3},2)\gamma_1^+(\bar{3},2)\gamma_2^-(\underbar{3},2)\dots,\label{eq:ASL_3_2_plus}\\
            &\hspace{-2mm}\dots(\underbar{3},2)\beta_2^-(\bar{3},2)\beta_1^+(\underbar{3},2)\alpha^-(\bar{3},2)\gamma_1^-(\underbar{3},2)\gamma_2^+(\bar{3},2)\dots,\label{eq:ASL_3_2_minus}
        \end{align}
    \end{subequations}
    where $\alpha^+\in\mathcal{A}_3^+$, $\alpha^-\in\mathcal{A}_3^-$, $\beta_i^+$, $\gamma_i^+\in\mathcal{A}_2^+$, $\beta_i^-$, and $\gamma_i^-\in\mathcal{A}_2^-$. The term ``$\dots$'' means that more than one tangent markers can exist before or after the given sequence.
\end{theorem}

\begin{proof}
    Without loss of generality, we first prove that an ASL $\mathfrak{L}$ containing $\alpha^+\in\mathcal{A}_3^+$ must take the form of \eqref{eq:ASL_3_2_plus}. Denote the intervals of $\alpha^+$ and $\gamma_i^\pm$ are $(t_0,t_1)$ and $(t_{i},t_{i+1})$, respectively. Similar to the proof of Theorem \ref{thm:3order_inf_position}, we can derive that $\lambda_3(t)<0$ remains constant during arcs $\alpha^+$, i.e., $t\in(t_{0},t_{1})$. Furthermore, $\lambda_1(t_1)<0$, $\lambda_2(t_1)>0$. If no $(\underbar{3},2)$ occurs at $t_1$, then $\forall t>t_1$, $\lambda_3(t)\leq \lambda_3(t_1^-)$ holds; hence, $\lambda_1(t)<0$ and $u(t)\equiv\bar{u}$. So the trajectory ends with $\alpha^+$.

    If $(\underbar{3},2)$ occurs at $t_1$, then $\lambda_3(t_1^+)\geq\lambda_3(t_1^-)$ holds due to \eqref{eq:junction_condition}. If $\lambda_3(t_1^+)>0$ holds, then $\lambda_2(t)$ decreases monotonically in $(t_1,t_2)$; hence, $\lambda_2(t)$ crosses 0 at most one time in $(t_1,t_2)$. Similarly, $\lambda_1(t)$ crosses 0 at most one time before the next junction time, resulting in arcs $\gamma_1^+\in\mathcal{A}_2^+$. Note that velocity constraints $\underbar{x}_2\leq x_2\leq \bar{x}_2$ are not allowed to be active in $(t_1,t_2)$ since $\lambda_2(t_1)>0$. Therefore, if a junction occurs at $t_2$, then $(\bar{3},2)$ or $(\underbar{3},2)$ occur at $t_2$. By induction, it can be derived that the ASL $\mathfrak{L}$ is in the form of \eqref{eq:ASL_3_2_plus}. For the same reason, if $\mathfrak{L}$ contains $\alpha^-\in\mathcal{A}_3^-$, then $\mathfrak{L}$ is in the form of \eqref{eq:ASL_3_2_minus}.

    Assume that $\mathfrak{L}$ does not contain $\alpha\in\mathcal{A}_3^+\cup\mathcal{A}_3^-$ and is not in the form of \eqref{eq:ASL_3_2}. \textit{Step 1.} One can prove that $\mathfrak{L}$ does not contain $\bar{2}$- or $\underbar{2}$-arcs through costate analysis. Then, $\mathfrak{L}$ should contain tangent markers; otherwise, $\mathfrak{L}\in\mathcal{A}_3^+\cup\mathcal{A}_3^-$, which contradicts the assumption. \textit{Step 2.} Denote $\mathfrak{L}=\mathfrak{r}_0\mathfrak{t}_1\mathfrak{r}_1\mathfrak{t}_2\dots\mathfrak{t}_N\mathfrak{r}_N$ where $\mathfrak{t}_i$ is a tangent marker. The time interval of $\mathfrak{r}_i$ is $(t_{i},t_{i+1})$. It can be proved that $\mathfrak{t}_i\not=\mathfrak{t}_{i+1}$ and $\mathfrak{r}_i\in\mathcal{A}_2^+\cup\mathcal{A}_2^-$. \textit{Step 3.} Assume that $\mathfrak{r}_i,\mathfrak{r}_{i+1}\in\mathcal{A}_2^+$. Then, $\lambda_1$ crosses 0 decreasingly at $t_{i+1}$ and $\mathfrak{t}_i=(\underbar{3},2)$. Through costate analysis, one can prove that $\mathfrak{L}$ is in the form of \eqref{eq:ASL_3_2_plus}, where we let $\alpha^+=\mathfrak{r}_i\bar{0}$ with an additional $\bar{0}$-arc over a zero time duration. Therefore, if $\mathfrak{r}_i,\in\mathcal{A}_2^+$, then $\mathfrak{r}_{i+1}\in\mathcal{A}_2^-$. Similarly, if $\mathfrak{r}_i,\in\mathcal{A}_2^-$, then $\mathfrak{r}_{i+1}\in\mathcal{A}_2^+$. \textit{Step 4.} Without loss of generality, assume that $\mathfrak{r}_0\in\mathcal{A}_2^+$. If $\mathfrak{t}_1=(\bar{3},2)$, then $\mathfrak{L}$ is in the form of \eqref{eq:ASL_3_2_minus}, where we let $\alpha^-=\underbar{0}\mathfrak{r}_0$ with an additional $\underbar{0}$-arc over a zero time duration. Similarly, if $\mathfrak{t}_1=(\underbar{3},2)$, then $\mathfrak{L}$ is in the form of \eqref{eq:ASL_3_2_minus}. The above analysis leads to a contradiction.

\end{proof}

\begin{remark}
    While our previous work \cite{wang2025time} provides a complete classification of arcs and tangent markers, it failed to enumerate all possible ASLs, i.e., the allowed connection of these arcs and tangent markers. Theorem \ref{thm:3order_full_box_constraints} fills this gap by presenting, for the first time, all possible ASLs form for triple integrator with full box constraints.
\end{remark}

It is significant to determine the condition under which a tangent marker exists. Intuitively, if the position constraints are sufficiently relaxed, i.e., $\vx_0$ and $\vx_\f$ lie far away from the position constraint boundary, then the position constraints are inactive throughout the time-optimal trajectory. In this case, no tangent marker arises.

According to the analysis on Fig. \ref{fig:3order_manifold_constrained_position}, switching surfaces like $\underbar{0}(\bar{3},2)\underbar{0}\bar{0}\underbar{0}$ can be regarded as a deformation of surfaces like $\underbar{0}\bar{0}$ near the position constraint boundary. In other words, the tangent marker $(\underbar{3},2)$ exists when original switching surfaces like $\underbar{0}\bar{0}$ and $\bar{0}\underbar{0}$ intersects \{$x_3=\underbar{x}_3$, $x_2=0$, $x_1\geq0$\}, inducing switching surfaces like $\underbar{0}(\underbar{3},2)\underbar{0}\bar{0}\underbar{0}$ and $\underbar{0}\bar{0}(\underbar{3},2)\bar{0}\underbar{0}$, respectively. Similar analysis can be applied to $(\bar{3},2)$.

\begin{algorithm}[!t]
    \caption{Determining the Relative Direction}
    \label{alg:direction}
    \begin{algorithmic}[1]
        \REQUIRE $\vx_0$, $\vx_\f$, $\bar{\vx}$, $\underbar{\vx}$, $\bar{u}$, $\underbar{u}$.
        \ENSURE The relative direction from $\vx_0$ to $\vx_\f$.
        \STATE Find all BBS controls from $\vx_{0,1:2}$ to $\vx_{\f,1:2}$, denoted by $\hat{u}_i(t)$, $t\in[0,t_{\f,i}]$, $i=1,2,\dots,I$.\label{line:find_bbs_control_2nd}
        \STATE Generate $\vx_i(t)$ by integrating $\hat{u}_i(t)$ from $\vx_0$.
        \STATE Identify the number of $i$, s.t. $\vx_{i,3}(t_{\f,i})\geq\vx_{\f,3}$.
        \RETURN $\vx_0$ is higher than $\vx_\f$ \textbf{if} $i$ is odd; otherwise, $\vx_0$ is lower than $\vx_\f$.
    \end{algorithmic}
\end{algorithm}

\subsection{Algorithm}\label{subsec:algorithm}

This section presents an efficient algorithm for solving the 3rd-order problem \eqref{eq:problem} with full box constraints. 

For the first step, the relative direction from $\vx_0$ to $\vx_\f$ is determined in Algorithm \ref{alg:direction}, where position constraints are temporarily ignored. In other words, switching surfaces in Fig. \ref{fig:3order_manifold_inf_position}(b) are considered. $\vx_0$ is \textit{higher} than $\vx_\f$ if $\vx_0$ lies on the same side as the direction of $(0,0,+\infty)$ relative to the surface $\bar{0}\underbar{0}\cup\bar{0}\bar{1}\underbar{0}\cup\underbar{0}\bar{0}\cup\underbar{0}\underbar{1}\bar{0}$; otherwise, $\vx_0$ is \textit{lower} than $\vx_\f$.

In Algorithm \ref{alg:direction}, Line \ref{line:find_bbs_control_2nd} aims to deal with the case where the surface $\bar{0}\underbar{0}\cup\bar{0}\bar{1}\underbar{0}\cup\underbar{0}\bar{0}\cup\underbar{0}\underbar{1}\bar{0}$ is not single-valued w.r.t. $\left(x_1,x_2\right)$. Without loss of generality, assume that $\vx_{0,1:2}$ is lower than the curve $\bar{0}\cup\underbar{0}$ in the 2nd-order problem, corresponding to the optimal control with ASL $\bar{0}\underbar{0}$ or $\bar{0}\bar{1}\underbar{0}$. Then, $I=3$ holds if $x_{0,1}\geq0$, $x_{\f,1}>0$, and $x_{0,2}-\frac{x_{0,1}^2}{2\underbar{u}}\geq x_{\f,2}+\frac{x_{\f,1}^2}{2\bar{u}}$. The additional two BBS controls are with ASLs $\underbar{0}\bar{0}$ or $\underbar{0}\underbar{1}\bar{0}$, which are not optimal but feasible. The multiple BBS solutions corresponds to the multiple-valued surface $\bar{0}\underbar{0}\cup\bar{0}\bar{1}\underbar{0}\cup\underbar{0}\bar{0}\cup\underbar{0}\underbar{1}\bar{0}$. Similar analysis can be applied to the case where $\vx_{0,1:2}$ is higher than the curve $\bar{0}\cup\underbar{0}$. For all cases, $I\in\{0,1,3\}$, where $I=0$ means that the problem is infeasible.

Then, Algorithms \ref{alg:noposition} and \ref{alg:withposition} solve the 3rd-order optimal control problem \eqref{eq:problem} without and with position constraints, respectively. The connection of two points or a point and a curve can be solved by a system of polynomial equations w.r.t. the durations $\vt=(t_i)_{i=1}^M$ of arcs. Based on the Gr\"obner basis, the system can be solved by a polynomial equation w.r.t. $t_1$ of order $\leq 6$. These polynomial equations can be generated once and for all using Macaulay2 \cite{Macaulay2}, with boundary conditions and constraints incorporated as parametric coefficients.

\begin{remark}
    Algorithms in this section extend the completeness of the methods presented in our previous work \cite{wang2025time}. In \cite[Definition 10]{wang2025time}, a concept of proper position is introduced to determine the relative direction from $\vx_0$ to $\vx_\f$. However, such approach would fail when the 2-dimensional projection of switching surfaces is not single-valued w.r.t. $\left(x_1,x_2\right)$, e.g., Fig. \ref{fig:3order_manifold_inf_position}(b). Algorithm \ref{alg:direction} addresses this issue by finding all BBS controls in the 2nd-order problem, and thereby guarantees the correctness of the relative direction. Furthermore, Algorithm \ref{alg:withposition} resolves another incompleteness issue in \cite{wang2025time}, where $\gamma_i^\pm$, $i\geq0$, and $\beta_i^\pm$, $i\geq1$, in \eqref{thm:3order_full_box_constraints} is not considered.
\end{remark}

\begin{algorithm}[!t]
    \caption{Solving 3rd-order Optimal Control Problem \eqref{eq:problem} without Position Constraints}
    \label{alg:noposition}
    \begin{algorithmic}[1]
        \REQUIRE $\vx_0$, $\vx_\f$, $\bar{\vx}$, $\underbar{\vx}$, $\bar{u}$, $\underbar{u}$. ($\bar{x}_3=+\infty$ and $\underbar{x}_3=-\infty$.)
        \ENSURE Optimal solution $\vx^*(t)$, $u^*(t)$, and $t_\f^*$.
        \IF {$\vx_0$ is higher than $\vx_\f$}
            \STATE Apply Algorithm \ref{alg:noposition} from $-\vx_0$ to $-\vx_\f$, with constraints $-\bar{\vx}\leq\vx\leq-\underbar{\vx}$ and $-\bar{u}\leq u\leq-\underbar{u}$, resulting in $\hat\vx^*(t)$, $\hat u^*(t)$, and $\hat{t}_\f^*$.
            \RETURN $\vx^*(t)\leftarrow-\hat\vx^*(t)$, $u^*(t)\leftarrow-\hat u^*(t)$, $t_\f^*\leftarrow\hat{t}_\f^*$.
        \ENDIF
        \STATE Drive $\vx_0$ and $\vx_\f$ to the $\bar{2}$-curve with 2nd-order optimal controls forwardly and backwardly, respectively, resulting in $\vx_1(t)$, $t\in[0,t_{\f1}]$ and $\vx_2(t)$, $t\in[-t_{\f2},0]$.
        \IF {$x_{1,3}(t_{\f1})\leq x_{2,3}(-t_{\f2})$}
            \RETURN The connection of $\vx_1(t)$, $\bar{2}$, and $\vx_2(t)$.
        \ENDIF
        \STATE Determine the arc $\mathfrak{r}_1$ of $\vx_1(t)$ whose two ends $\vx^{(1)}$ and $\vx^{(2)}$ are lower and higher than $\vx_\f$, respectively.
        \STATE Connect $\vx^{(1)}$ and $\vx_\f$ with ASL $\mathfrak{L}=\mathfrak{r}_1\underbar{0}\bar{0}$ or $\mathfrak{r}_1\underbar{0}\underbar{1}\bar{0}$.
        \RETURN The connection of $\vx_1(t)$ and $\mathfrak{L}$.
    \end{algorithmic}
\end{algorithm}

\section{Numerical Experiments}

\subsection{Setup}

\textbf{Baselines.} (a) \textit{Ruckig} \cite{berscheid2021jerk} (community version): a state-of-the-art for 3rd-order problem \eqref{eq:problem}, where the offline planning interface is applied. (b) \textit{CasADi} \cite{andersson2019casadi}: a package for numerical optimization based on IPOPT \cite{wachter2006implementation}. (c) \textit{SCP} \cite{leomanni2022time}: a method based on the sequential convex programming, where optimization problems are solved by Gurobi \cite{gurobi}. 1000 intervals are applied in CasADi and SCP. All methods are implemented in C++ 20. Each method terminates if the computational time exceeds 150\,s or the solution is convergent.

\textbf{Metrics.} (a) Computational time $T_\text{c}$: all experiments were preformed on a computer with an Intel${}^\text{\textregistered}$ Core${}^\text{TM}$ Ultra 9 Processor 275HX. (b) Terminal time $t_\f$. (c) Success rate $R_\text{s}$ to obtain a feasible solution. An experiment is successful if the exceedance of each constraint is no more than $10^{-3}$ and the open-loop error $E_\text{o}=\sqrt{\frac13\sum_{k=1}^{3}\left(\frac{x_i(t_\f)-x_{\f,i}}{\sup x_i(t)-\inf x_i(t)}\right)^2}\leq 10^{-4}$, where $\vx(t_\f)$ is the solved terminal state integrated from $\vx_0$ with $u(t)$, and $\vx_\f$ is the desired terminal state. 

\subsection{Numerical Results}

\begin{algorithm}[!t]
    \caption{Solving 3rd-order Optimal Control Problem \eqref{eq:problem} with Full Box Constraints}
    \label{alg:withposition}
    \begin{algorithmic}[1]
        \REQUIRE $\vx_0$, $\vx_\f$, $\bar{\vx}$, $\underbar{\vx}$, $\bar{u}$, $\underbar{u}$. ($\bar{x}_3$, $\underbar{x}_3$ are allowed to be finite.)
        \ENSURE Optimal solution $\vx^*(t)$, $u^*(t)$, and $t_\f^*$.
        \STATE Apply Algorithm \ref{alg:noposition} and obtain the optimal solution $\hat{\vx}(t)$, $\hat{u}(t)$, and $\hat{t}_\f$ without position constraints.
        \IF {$\hat{x}_3(t)$ satisfies the position constraints}
            \RETURN $\vx^*(t)\leftarrow\hat\vx^*(t)$, $u^*(t)\leftarrow\hat u^*(t)$, $t_\f^*\leftarrow\hat{t}_\f^*$.
        \ENDIF
        \IF {The constraint $\hat{x}_3\geq\underbar{x}_3$ is exceeded}
            \STATE Connect $\vx_0$ or $\vx_\f$ with $(\underbar{3},2)$ based on Theorem \ref{thm:3order_full_box_constraints}, resulting in the tangent point $\vx_{(\underbar{3},2)}$.\label{line:connect_tangent_point}
            \STATE Connect $\vx_{(\underbar{3},2)}$ and $\vx_\f$ or $\vx_0$ based on Algorithm \ref{alg:withposition} recursively.
            \RETURN The connection of the above two profiles \textbf{if} the process succeeds; otherwise, \textbf{return} infeasibility. \label{line:connect_tangent_point_return}
        \ENDIF
        \IF {The constraint $\hat{x}_3\leq\underbar{x}_3$ is exceeded}
            \STATE Apply the process symmetrical to Lines \ref{line:connect_tangent_point}--\ref{line:connect_tangent_point_return} and \textbf{return} the result.
        \ENDIF
    \end{algorithmic}
\end{algorithm}

\begin{table*}[!t]
    \centering
    \caption{Quantitative Results. The values are presented as mean (standard deviation).}
    \label{tab:results}
    \begin{tabular}{c|cccc|ccc}
        \hline
        Case & \multicolumn{4}{c|}{Case 1 (Easy)} & \multicolumn{3}{c}{Case 2 (Difficult)} \\ \hline
        Method & Ours & Ruckig & CasADi & SCP & Ours & CasADi & SCP \\ \hline
        Success Rate $R_\text{s}$ & \textbf{100.0\%} & \textbf{100.0\%} & 77.7\% & 99.7\% & \textbf{100.0\%} & 97.7\% & 34.1\% \\ \hline
        Computational Time $T_\text{c}$ (s) & \begin{tabular}[c]{@{}c@{}}2.38$\times$10${}^{-6}$\\ (1.38$\times$10${}^{-6}$)\end{tabular} & \textbf{\begin{tabular}[c]{@{}c@{}}7.00$\times$10${}^{-7}$\\ (4.26$\times$10${}^{-7}$)\end{tabular}} & \begin{tabular}[c]{@{}c@{}}0.28\\ (0.21)\end{tabular} & \begin{tabular}[c]{@{}c@{}}7.28\\ (4.24)\end{tabular} & \textbf{\begin{tabular}[c]{@{}c@{}}1.25$\times$10${}^{-5}$\\ (4.24$\times$10${}^{-6}$)\end{tabular}} & \begin{tabular}[c]{@{}c@{}}4.76\\ (17.70)\end{tabular} & \begin{tabular}[c]{@{}c@{}}110.06\\ (43.40)\end{tabular} \\ \hline
        Normalized Terminal Time $\frac{t_\f}{t_{\f,\text{ours}}}$ & \textbf{\begin{tabular}[c]{@{}c@{}}1.00\\ (0.00)\end{tabular}} & \textbf{\begin{tabular}[c]{@{}c@{}}1.00\\ (0.00)\end{tabular}} & \begin{tabular}[c]{@{}c@{}}1.07\\ (1.01)\end{tabular} & \begin{tabular}[c]{@{}c@{}}1.03\\ (0.84)\end{tabular} & \textbf{\begin{tabular}[c]{@{}c@{}}1.00\\ (0.00)\end{tabular}} & \begin{tabular}[c]{@{}c@{}}1.03\\ (0.29)\end{tabular} & \textbf{\begin{tabular}[c]{@{}c@{}}1.00\\ (1.36$\times$10${}^{-3}$)\end{tabular}} \\ \hline

    \end{tabular}
\end{table*}

Two cases are designed in this section. Case 1 considers symmetric jerk constraints, i.e., $\bar{u}=-\underbar{u}$, without position constraints, i.e., $\bar{x}_3=+\infty$, $\underbar{x}_3=-\infty$. Case 2 involves asymmetric jerk constraints and position constraints. For each case, 1000 random instances are generated, where $\bar{u},-\underbar{u}\in[0.9,1.1]$, $\bar{x}_1,-\underbar{x}_1\in[0.7,1.3]$, $\bar{x}_2,-\underbar{x}_2\in[1.0,2.0]$. Boundary states $\vx_0$ and $\vx_\f$ are randomly generated in the feasible region. For Case 2, position constraints are active, i.e., tangent markers $(\underbar{3},2)$ or $(\bar{3},2)$ occur.

For Case 1, the results of Ruckig are close to ours, as shown in Table \ref{tab:results}. The computational times of Ruckig and our method are both shorter than 2.5$\mu\text{s}$, representing a reduction of at least 5 orders of magnitude compared to the other two baselines. For the example shown in Fig. \ref{fig:compare_trajectories}(a), both CasADi and SCP converged to local optima with an ASL $\underbar{0}\underbar{1}\bar{0}\bar{1}\underbar{0}$, which satisfy PMP but are not globally optimal. In contrast, our method plans the time-optimal trajectory with an ASL $\bar{0}\underbar{0}\bar{0}$, saving 90.9\% of terminal time. Among the four methods, our method and Ruckig achieve a 100\% success rate.

For Case 2, Ruckig fails due to the asymmetric jerk constraints and the position constraints. Due to the introduction of position constraints, both the computational efficiency and the success rate of the SCP method deteriorate significantly. CasADi achieves a higher success rate than that in Case 1 since the terminal time in Case 2 is shorter than that in Case 1. Our method achieves a 100\% success rate and saves 5-orders-of-magnitude computational time compared to CasADi and SCP. For the example shown in Fig. \ref{fig:compare_trajectories}(b), our method plans the time-optimal trajectory with an ASL $\bar{0}\bar{1}\underbar{0}(\underbar{3},2)\underbar{0}\bar{0}\underbar{0}$, while CasADi fails to converge to a BBS solution. The terminal time of our method is reduced by 73.4\% compared to CasADi in this example.

\begin{figure}[!t]
    \centering
    \includegraphics[width=\linewidth]{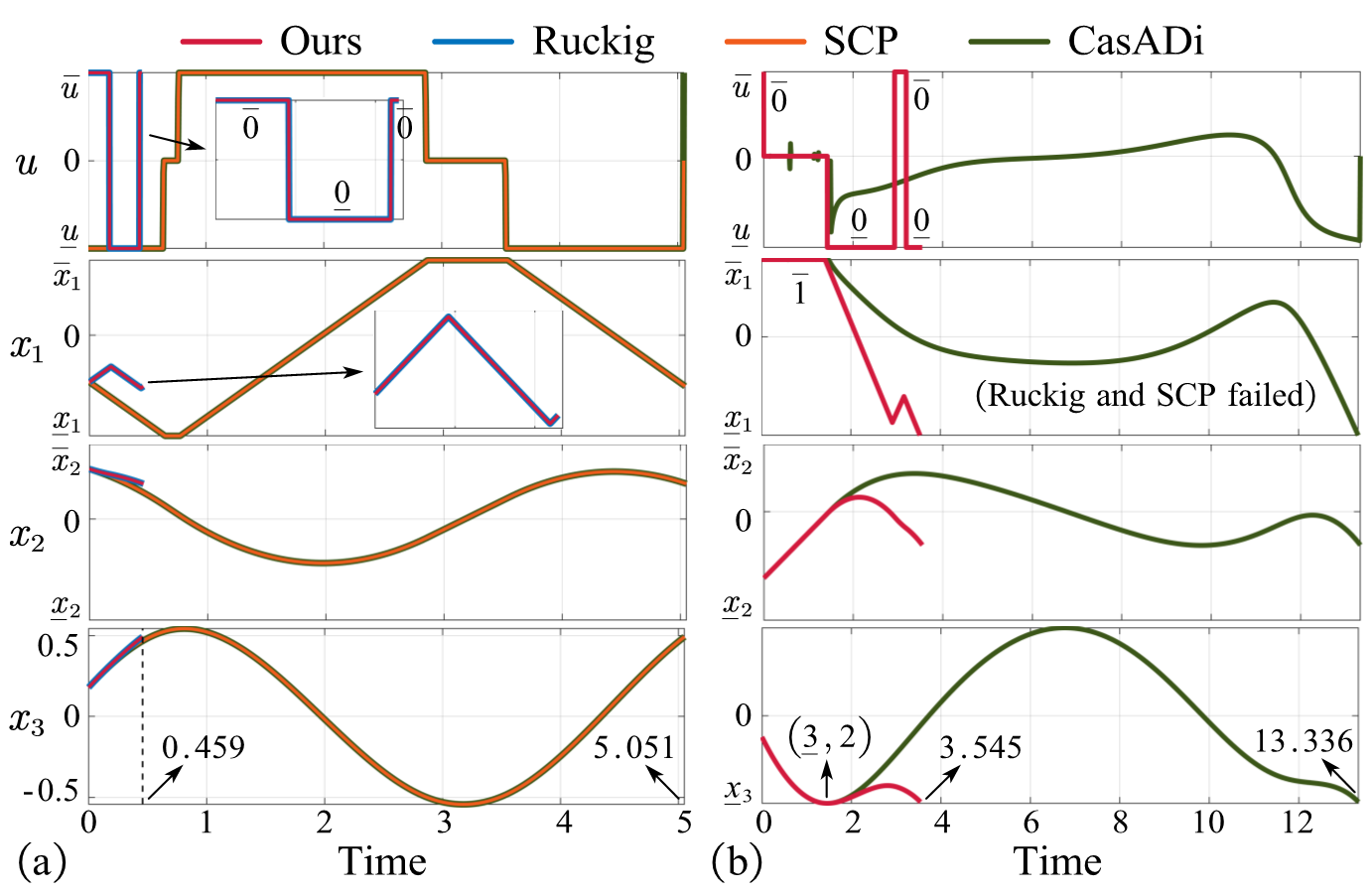}
    \caption{Comparison of trajectories planned by different methods. In (a), $\vx_0=(-0.582,0.797,0.333)$, $\vx_\f=(-0.634,0.561,0.647)$, $-1.035\leq u\leq1.035$, $(-1.246,-1.607,-\infty)\leq\vx\leq(0.925,1.191,+\infty)$. In (b), $\vx_0=(0.754,-1.101,-0.258)$, $\vx_\f=(-0.963,-0.553,-1.046)$, $-1.058\leq u\leq0.977$, $(-0.963,-1.800,-1.062)\leq\vx\leq(0.754,1.111,4.244)$.} 
    \label{fig:compare_trajectories}
\end{figure}

\section{Conclusion and Contribution}

This paper addressed the time-optimal control problem for triple integrator under full box constraints, which remains challenging with asymmetric constraints, non-stationary boundary states, and active position constraints. From a geometric perspective, this paper analytically characterized the time-optimal switching surfaces and divided the state space into several regions, where the optimal control can be represented by the same augmented switching law (ASL). The active condition of position constraints is also derived, resulting in a sequential feature of the ASL. An efficient algorithm is established to solve the problem. In the experiment, the proposed method solved all feasible problems with a 100\% success rate, whereas the baselines failed on some instances. The computational time, within approximately 10$\mu$s in our method, is reduced by at least 5 orders of magnitude compared to the baselines. The terminal time is reduced by more than 70\% in some cases since our method gets rid of the local minimum. To the best of the authors' knowledge, this is the first work to provide complete switching surfaces and an efficient algorithm for the 3rd-order problem, especially under non-stationary terminal conditions and active position constraints.

\bibliographystyle{myIEEEtran}
\bibliography{IEEEabrv,refs/ref}

\end{document}